\documentclass[11pt]{amsart}

\usepackage[normalem]{ulem}
\usepackage[svgnames,x11names]{xcolor}
\usepackage{amssymb,amsfonts}
\usepackage[pagebackref,colorlinks,urlcolor={DarkOliveGreen4},
linkcolor={DarkSlateBlue},
citecolor={Chocolate4}]{hyperref}
\usepackage[capitalize]{cleveref}
\usepackage{amssymb,textcomp}
\usepackage{graphicx}
\numberwithin{equation}{section}
\usepackage{bbm}
\usepackage{enumerate}
\usepackage[utf8]{inputenc}
\usepackage[T1]{fontenc}
\usepackage[mathscr]{eucal}
\usepackage{etoolbox}
\usepackage {tikz}
\usetikzlibrary {positioning}
\usepackage[a4paper, twoside=false, vmargin={2cm,3cm}, includehead]{geometry}
\usepackage{nameref}
\newtheorem{lemma}{Lemma}[section]
\newtheorem{theorem}[lemma]{Theorem}
\newtheorem*{theorem*}{Theorem}
\newtheorem{corollary}[lemma]{Corollary}
\newtheorem{question}{Question}
\newtheorem*{question*}{Open question}

\newtheorem*{proposition*}{Proposition}

\newtheorem*{problem*}{Problem}
\theoremstyle{definition}
\newtheorem{definition}[lemma]{Definition}
\newtheorem*{claim*}{Claim}
\newtheorem*{notation}{Notation}

\newtheorem{remark}{Remark}

\newcommand{\Mod}[1]{\ \mathrm{mod}\ #1}
\usepackage{framed}

\makeatletter
\theoremstyle{plain}
\newtheorem*{namedthm}{\namedthmname}
\newcounter{namedthm}
	

\newcommand{\E}{{\mathbb E}}

\newcommand{\N}{{\mathbb N}}
\renewcommand{\P}{{\mathbb P}}

\newcommand{\R}{{\mathbb R}}

\newcommand{\Cov}{\operatorname{Cov}}

\newtheorem{theoremL}{Theorem}

\newtheorem{CorollaryL}{Corollary}

\begin{document}

\title{A pointwise ergodic theorem along return times of rapidly mixing systems}



\author{Sebasti\'an Donoso}
\address{Departamento de Ingenier\'{\i}a Matem\'atica and Centro de Modelamiento Matem{\'a}tico, Universidad de Chile \& IRL 2807 - CNRS, Beauchef 851, Santiago, Chile}
\email{sdonosof@uchile.cl}

\author{Alejandro Maass}
\address{Departamento de Ingeniería Matemática, Centro de Modelamiento Matemático and Millennium Institute Center for Genome Regulation, Santiago, Chile, Universidad de Chile and IRL-CNRS 2807, Beauchef 851, Santiago, Chile.}
\email{amaass@dim.uchile.cl}

\author{Vicente Saavedra-Araya}
\address{Department of Mathematics, University of Warwick, Coventry, United Kingdom}	
\email{vicente.saavedra-araya@warwick.ac.uk}

\thanks{ 
All authors were partially funded by Centro de Modelamiento Matemático (CMM) FB210005, BASAL funds for centers of excellence from ANID-Chile. The first author was partially funded by ANID/Fondecyt/1241346. The second author was funded by the grant ICN2021-044 from the ANID Millennium Science Initiative.
}

\subjclass[2020]{Primary: 37A30; Secondary: 37A50.}

\begin{abstract} We introduce a new class of sparse sequences that are ergodic and pointwise universally $L^2$-good for ergodic averages. That is, sequences along which the ergodic averages converge almost surely to the projection to invariant functions. These sequences are generated randomly as return or hitting times in systems exhibiting a rapid correlation decay. This can be seen as a natural variant of Bourgain's Return Times Theorem.  As an example, we obtain that for any $a\in (0,1/2)$, the sequence $\{n\in\N:\ 2^ny\Mod{1}\in (0,n^{-a})\}$ is ergodic and pointwise universally $L^2$-good for Lebesgue almost every $y\in [0,1]$.  Our approach builds on techniques developed by Frantzikinakis, Lesigne, and Wierdl in their study of sequences generated by independent random variables, which we adapt to the non-independent case.
\end{abstract}

\maketitle
\small
\normalsize
\section{Introduction}

We say that $(X,\mathcal{X},\mu,T)$ is a \emph{measure-preserving system} (m.p.s. for short) if $(X,\mathcal{X},\mu)$ is a probability space and $T\colon X\to X$ is a measure-preserving transformation, i.e., $T$ is measurable with respect to $\mathcal{X}$  and $\mu(T^{-1}A)=\mu(A)$ for all $A\in \mathcal{X}$. 
The measure-preserving system is said to be \emph{ergodic} if every set in the $\sigma$-algebra $\mathcal{I}(T):=\{A\in \mathcal{X}:\ T^{-1}A=A \Mod{\mu}\}$ has either zero or full measure. 
In this paper the natural numbers are $\N = \{1, 2, 3, \ldots\}$. A sequence of integers $(a_n)_{n\in \N}$ is said to be \emph{pointwise universally $L^p$-good} if for any m.p.s. $(X,\mathcal{X},\mu,T)$ and $f\in L^p(\mu)$,  the ergodic averages 
\begin{equation}
    \dfrac{1}{N}\sum_{n=1}^Nf\left(T^{a_n}x\right)\label{eqn:averages}
\end{equation}
converge for almost every $x\in X$. 
The sequence $(a_n)_{n\in \N}$ is {\em ergodic} if the $L^2$-limit of \eqref{eqn:averages} exists and equals  $\mathbb{E}_\mu(f|\mathcal{I}(T))$.
Under this terminology, the classical Birkhoff's ergodic theorem establishes that $a_n=n$ is pointwise universally $L^1$-good and ergodic. 

A classic problem in ergodic theory lies in finding/building interesting sequences or families of sequences that are universally good and/or ergodic.  

Bourgain's Return Times Theorem \cite{ReturnTimes} illustrates that the notion of \emph{return times} can be used to randomly generate universally good sequences.

\begin{theorem}[Return Times Theorem \cite{ReturnTimes}] \label{Return times ergodic theorem} Let $(Y,\mathcal{Y},\nu,S)$ be an ergodic m.p.s. and $E\in \mathcal{Y}$ such that $\nu(E)>0$. For almost every $y\in Y$, the return-times sequence\footnote{ Strictly speaking, for $y\notin E$, the first time $y$ enters $E$ is a {\em hitting time} rather than a return time. However, for simplicity and consistency with the literature, we will use the term ``return time'' to refer to both scenarios. See \cite{Survey_Assani_Presser} for a survey about results of this nature. } $$\Lambda_y:=\{n\in \N:\ S^ny\in E\}$$ is pointwise universally $L^1$-good.
\end{theorem}

Note that the ergodic theorem implies that for almost every $y\in Y$,
\[d(\Lambda_y):=\lim_{N\to \infty}\dfrac{|\{1\leq n\leq N:\ S^ny\in E\}|}{N}=\nu(E)>0,\]
showing that the sequence $\Lambda_y$ is not sparse\footnote{Here we use the term sparse to refer to any subset $\Lambda\subseteq \N$ such that $\limsup_{N\to \infty} |\Lambda \cap [1,N]|/N =0$.}. It is also worth noting that $\Lambda_y$ is not necessarily ergodic, as in a finite system, $\Lambda_y$ can take the form $k\N$ for $k\in \N$. 

Finding sparse sequences that are universally good and/or ergodic presents significant challenges. In a groundbreaking work, Bourgain \cite{Bourgain_Pointwise_89} showed that if $q$ is a polynomial with integer coefficients and degree greater than one, the sparse sequence $a_n=q(n)$ is pointwise universally $L^p$-good for any $p>1$, though not necessarily ergodic (see \cite{krause2023pointwiseergodictheoryexamples} for an exposition). More recently,  ergodic sparse sequences have been constructed in the context of single and multiple ergodic averages (see \cite{Boshernitzan_Kolesnik_Quas_Wierdl_erg_sequences:2005, Donoso_Ferre_Koutsogiannis_Sun_multicorr_joint_erg:2024,Donoso_Koutsogiannis_Kuca_Tsinas_Sun_multiple_Hardy,Donoso_Koutsogiannis_Sun_seminorms_polynomials_joint_ergodicity:2022,Donoso_Koutsogiannis_Sun_joint_erg_poly_growth:2023,Frantzikinakis_multidim_Szemeredi_Hardy:2015,Frantzikinakis_joint_ergodicity_sequences:2023,Frantzikinakis_Kuca_joint_erg_comm_poly:2025}). On recent results on the pointwise convergence of multiple ergodic averages, see \cite{Donoso_Sun_pointwise_convergence_multiple_averages:2018, Huang_Shao_Ye_pointwise_convergence_models:2019,KMPW,
krause2025unified,Krause_Tao_Mirek}.

To address the challenge of finding sparse sequences that are both universally good and ergodic, Bourgain introduced in \cite{Bourgain1988} a method to randomly generate such sequences.  
\begin{definition}
    Let $(\Omega,\mathcal{F},\mathbb{P})$ be a probability space and let $(X_n)_{n\in \N}$ be a sequence of independent random variables defined in this space taking values in $\{0,1\}$. For $\omega\in \Omega$, we define the \emph{random sequence of integers generated by $(X_n)_{n\in \N}$} as
    \begin{equation}
        a_n(\omega):=\inf \Big\{k\in \N:\ X_1(\omega)+\cdots+X_k(\omega)=n\Big\}.\label{eq:random_sequence}
    \end{equation}\label{def:random_sequence}
\end{definition}
That is, $a_n(\omega)$ represents the $n$-th time that the trajectory $(X_n(\omega))_{n\in\N}$ hits the state $1$. Note that in order for the sequence in \eqref{eq:random_sequence} to be well defined, it is necessary that $X_n(\omega)=1$ infinitely many times.
Sequences of this kind are called \emph{random sequences of integers}, and the averages like \eqref{eqn:averages} when $a_n=a_n(\omega)$ are usually called \emph{random ergodic averages} (for some results in this direction, see \cite{Frantzikinakis_Lesigne_Wierdl,Frantzikinakis_Lesigne_Wierdl_Szemeredi,Krause_ZorinKranich_1,Krause_ZorinKranich_2,Rosenblatt}). 
Commonly, one restricts to the case in which there is $a\in (0,1)$ such that $\mathbb{P}(X_n=1)=n^{-a}$ for all $n\in \N$. 
In such a case, for almost every $\omega\in \Omega$, the random sequence is sparse, since  $a_n(\omega)$ behaves, in a certain sense, similarly to $\lfloor n^{1/(1-a)}\rfloor$ (see \cite[Lemma A.6]{Frantzikinakis_Lesigne_Wierdl} for a precise statement).
The next result is the aforementioned Bourgain's theorem. 
\begin{theorem}[\cite{Bourgain1988}]
     Let $(\Omega,\mathcal{F},\mathbb{P})$ be a probability space and let $(X_n)_{n\in \N}$ be a sequence of independent random variables defined in this space taking values in $\{0,1\}$. Suppose that there exists $a\in (0,1)$ such that $\mathbb{P}(X_n=1)=n^{-a}$ for all $n\in \N$\footnote{This condition can be replaced by the existence of $\delta$>1 such that $\lim_{n\to \infty}n\mathbb{P}(X_n)/(\log \log n)^\delta=\infty$.}.  Then, for any $p>1$, for almost every $\omega \in \Omega$, the sequence $(a_n(\omega))_{n\in \N}$ is pointwise universally $L^p$-good and ergodic.\label{thm:random_ergodic_averages}
\end{theorem}
The case $p=1$ was later shown by LaVictoire \cite{LaVictoire} for $a\in (0,1/2)$. 
\medskip

 In this paper, we give a version of \cref{thm:random_ergodic_averages} in the context of non-independent random variables. 
 From that result, we retrieve a natural variant of the Return Times Theorem (\cref{Return times ergodic theorem}), in which we allow the sets to vary with $n$. This is done by considering sequences of the form $\{n\in \N:\ S^ny\in E_n\}$ for some measurable sets $(E_n)_{n\in \N}$, with $\nu(E_n)=n^{-a}$. 
 This method allows us to build new families of sequences that are sparse, pointwise universally good and ergodic. The definition of return times for a sequence of sets is as follows.
 
\begin{definition}
    Let $(Y,\mathcal{Y},\nu,S)$ be an ergodic m.p.s. and $\mathcal{E}=(E_n)_{n\in\N}$ be a sequence of measurable subsets of $Y$. We define the entry time of $y$ to $\mathcal{E}$ as $$r_1(y,\mathcal{E}):=\inf\{k\in \N:\ S^ky\in E_k\}.$$ Inductively, for $n>1$ the $n$-th return time is defined by
    $$r_n(y,\mathcal{E}):=\inf\{k>r_{n-1}(y,\mathcal{E}):\ S^ky\in E_k\}.$$
    Alternatively, $(r_n(y,\mathcal{E}))_{n\in \N}$ is the enumeration of the set $\{n\in \N: S^ny \in E_n  \}$ in increasing order. 
    The sequence $(r_n(y,\mathcal{E}))_{n\in \N}$ is called the sequence of return times (or hitting times) of $y$ to $\mathcal{E}$.
    \label{def:return_times}
\end{definition}

It is worth mentioning that the sequence of return times has been widely studied in systems that exhibit rapidly mixing properties. For instance, if $p>1$, $S(y)=p y\Mod{1}$ and $\lambda$ is the Lebesgue measure, an important work of Philipp \cite{Philipp_Walter} shows that for any sequence of intervals $(I_n)_{n\in \N}\subseteq [0,1)$,

\begin{equation}
\lim_{N\to\infty}\dfrac{|\{1\leq n\leq N:\ S^ny\in I_n\}|}{\sum_{n=1}^N\lambda(I_n)}=1\label{sBC_}
\end{equation}
for Lebesgue almost every $y\in [0,1)$. In such a case, the sequence $(I_n)_{n\in\N}$ is said to be \emph{strongly Borel-Cantelli}. Further results on strongly Borel-Cantelli sequences can be found in \cite{NonUniformly_BorelCantelli,Kim,Kleinbock_Margulis,LeVeque,Persson_Rodriguez}. In addition, the set of \emph{eventually always hitting points} has recently been studied in dynamical systems with rapidly mixing properties \cite{Ganotaki_Persson,Kirsebom_Kunde_Persson,Kleinbock_Konstantoulas_Richter}. 
\medskip

Before stating our main result, note that both $(\Omega,\mathcal{F},\mathbb{P})$ and $(Y,\mathcal{Y},\mu)$ represent probability spaces, but we make the distinction in notation to refer to a general probability space and to the specific case of a measure-preserving system.  Sequences in \cref{def:random_sequence} and \cref{def:return_times} are related. In fact, it turns out that the sequence $(r_n(y,\mathcal{E}))_{n\in \N}$ can be understood as the random sequence from \cref{def:random_sequence} associated with the random variables $X_n(y):=\mathbbm{1}_{S^{-n}E_n}(y)$ when $(\Omega,\mathcal{F},\mathbb{P})=(Y,\mathcal{Y},\mu)$, however, those random variables are generally not independent. We introduce the following class of sequences $(E_n)_{n\in\N}$ that will enable us to control the dependence of such random variables.
\begin{definition}
Let $(Y,\mathcal{Y},\nu,S)$ be a m.p.s. The collection $\mathcal{E}:=(E_n)_{n\in \N}$ of measurable subsets of $Y$ has a \emph{decay of correlation against $L^1(\nu)$ with rate $\rho\colon\N\to \mathbb{R}_{+}$} if there exists $C>0$ such that
    \[\left|\int_{E_n} g\circ S^md\nu-\nu(E_n) \cdot \int_Y gd\nu \right|\leq C\rho(m)\|g\|_{L^{1}}\]
    for all $n,m\in\N$ and $g\in L^1(\nu)$ supported on $\bigcup_{n\in\N}E_n$. \label{def:correlations}
\end{definition} As we shall see, examples of sequences satisfying the above property arise naturally in dynamical systems exhibiting certain decay of correlations.
\medskip

In this paper, we obtain the following result concerning the convergence of ergodic averages along return-time sequences. 

\begin{theoremL}
    Let $(Y,\mathcal{Y},\nu,S)$ be a m.p.s. and let $\mathcal{E}=(E_n)_{n\in \N}$ be a sequence of measurable sets such that:
    \begin{enumerate}
        \item  $\mathcal{E}$ has decay of correlation against $L^1(\nu)$ with rate $\rho$ decreasing and summable, and \label{C1}
        \item there exist $a\in (0,1/2)$ and $c>0$ such that $\nu(E_n)=cn^{-a}$.\label{C2}
       
    \end{enumerate}  Then, for almost every $y\in Y$, the sequence of return times $(r_n(y,\mathcal{E}))_{n\in \N}$ is pointwise universally $L^2$-good and ergodic. That is, for every m.p.s. $(X,\mathcal{X},\mu,T)$ and $f\in L^{2}(\mu)$,
    \begin{equation}
        \lim_{N\to\infty}\dfrac{1}{N}\sum_{n=1}^Nf(T^{r_n(y,\mathcal{E})}x)=\mathbb{E}_\mu(f|\mathcal{I}(T))(x)\label{eqn:thmA}
    \end{equation}
   for almost every $x\in X$.\label{thmA}
\end{theoremL}
\cref{thmA} follows from the more general \cref{main}, which can be understood as a non-independent version of \cref{thm:random_ergodic_averages}. Regarding the proof of \cref{main}, we follow the method developed in \cite{Frantzikinakis_Lesigne_Wierdl} by Frantzikinakis, Lesigne and Wierdl to study semi-random ergodic averages for independent random variables. The absence of independence in our context raises some difficulties in the proof, which we address by controlling the covariance of the random variables involved. 

Immediate applications of \cref{thmA} lie in considering systems $(Y,\mathcal{Y},\nu,S)$ that have a decay of correlation for functions of bounded variation against $L^1(\nu)$ with summable rate (see \cref{def:correlationsBV} and \cref{cor}). For instance, it is well known that the system defined by the multiplication by $p \mod 1$  in the unit interval with Lebesgue measure has an exponentially fast decay of correlations. Thus, we obtain the following.
\begin{CorollaryL}
    Let $a\in (0,1/2)$.  For any integer $p\geq 2$, and for almost every $y\in [0,1]$ (with respect to the Lebesgue measure), 
    the sequence
   $$
    \Big\{n\in\N:\ p^ny\Mod{1}\in (0,n^{-a})\Big\}$$
is pointwise universally $L^2$-good and ergodic. \label{corA}
\end{CorollaryL}
An exponentially fast decay of correlation for functions of bounded variation against $L^1$ also holds when considering the Gauss map with its invariant measure\footnote{The measure defined by $d\nu=\dfrac{1}{\ln(2)\cdot(1+x)}d\lambda$, where $\lambda$ is the Lebesgue measure, is invariant for the Gauss map $G$.}. It is also worth mentioning that condition (\ref{C2}) in \cref{thmA} can be easily replaced by the existence of $a\in (0,1/2)$ and $c>0$ such that $\nu(E_n)=cn^{-a}$ for every $n$ large enough. Hence, we conclude the following.
\begin{CorollaryL}
    Let $a\in (0,1/2)$ and let $G:[0,1)\mapsto [0,1)$ be the Gauss map, defined by $G(y)=1/y \mod{1}$ if $y\neq 0$, and $G(0)=0$.  
For any $b>1$, and for almost every $y\in [0,1)$ (with respect to the Lebesgue measure), the sequence
    $$\left\{n\in\N:\ G^ny\in (0, b^{n^{-a}}-1) \right\}$$
    is pointwise universally $L^2$-good and ergodic. \label{corB}
\end{CorollaryL}

Further examples of systems with exponentially fast decay of correlation for bounded variation against $L^1(\nu)$ (and therefore, applications of \cref{thmA}) include piecewise expanding maps \cite{Liverani}, some piecewise monotonic maps \cite{Liverani_Saussol_Vaienti} and some piecewise convex maps \cite{Cui}. In the multidimensional case, applications for \cref{thmA} can be obtained as a consequence of \cite[Theorem 1.1]{Eslami_Melbourne_Vaienti}, which provides decay of correlation for bounded variation against functions in $L^1$ supported on $[3/4,1]\times [0,1]$.

\vspace{0.7cm}

\textbf{Acknowledgements.} The authors express their gratitude to Joel Moreira for valuable discussions and to Ian Melbourne for providing useful references. We also thank the anonymous referees for their constructive feedback and valuable suggestions, which have improved the clarity of this paper.

\section{The class \eqref{P} and control of dependence}

 In this section, we introduce a class of sequences of sets and some auxiliary results that will be useful in the proof of our main result \cref{main} (and as a consequence, \cref{thmA}). We will focus on a purely probabilistic context, with no dynamics involved. 
 
We aim to study the sequence of random variables $X_n(\omega):=\mathbbm{1}_{E_n}(\omega)$, where $(E_n)_{n\in \N}$ is a sequence of measurable sets in a probability space $(\Omega,\mathcal{F},\mathbb{P})$, and find conditions to ensure that the random sequence of integers $(a_n(\omega))_{n\in\N}$ generated by the non-independent random variables $(X_n)_{n\in \N}$ is pointwise universally good. In this direction, we introduce the following class.

 \begin{definition} Let $(\Omega,\mathcal{F},\mathbb{P})$ be a probability space. We say that the collection of measurable sets $(E_n)_{n\in\N}$ satisfies the property \eqref{P} if there exists a decreasing sequence $\rho:\N\to \mathbb{R}_{+}$  such that $\displaystyle \lim_{n\to\infty}\rho(n)=0$ and  for every finite sequence of integers $1 \leq n_1<n_2<\cdots<n_k$,
 \begin{equation}
     \left|\mathbb{P}(E_{n_1} \cap \cdots \cap E_{n_k})-\mathbb{P}(E_{n_1})\mathbb{P}(E_{n_2}\cap \cdots \cap E_{n_k})\right|\leq \rho(n_2-n_1)\mathbb{P}(E_{n_2}\cap \cdots \cap E_{n_k}).  \tag{P}\label{P}
 \end{equation}
 \end{definition}
 \medskip
 Note that, if the collection $(E_n)_{n \in \N}$ is independent, then it satisfies \eqref{P} with rate $\rho=0$. Although condition \eqref{P} may seem artificial, numerous natural examples possess this property in the context of dynamical systems exhibiting decay of correlations for bounded variation against $L^1$.

 The next lemma allows us to conclude convergence almost everywhere of averages by proving convergence almost everywhere along lacunary subsequences\footnote{A sequence $(\gamma_n)_{n\in\N}$ is said to be lacunary if there exists $\gamma>1$ such that $\gamma_{n+1}/\gamma_n\geq \gamma$ for all $n\in \N$.}. The proof can be found in \cite[Corollary A.2]{Frantzikinakis_Lesigne_Wierdl}.

\begin{lemma} Let $(\Omega,\mathcal{F},\mathbb{P})$ be a probability space, for every $n\in \N$ consider a non-negative measurable function  $f_n\colon\Omega\to \mathbb{R}$, and let $(W_n)_{n\in \mathbb{N}}$ be an increasing sequence of positive numbers such that
\begin{equation}
    \lim_{\gamma\to 1^{+}}\limsup_{n\to \infty} \dfrac{W_{\lfloor\gamma^{n+1}\rfloor}}{W_{\lfloor\gamma^{n}\rfloor}}=1.\label{lemma1_h}
\end{equation}

For $N\in \N$ and $\omega\in \Omega$, define
\[A_N(\omega):=\dfrac{1}{W_N}\sum_{n=1}^{N} f_n(\omega).\]
Assume that there exist a function $f\colon \Omega\to \mathbb{R}$ and a real-valued sequence $(\gamma_k)_{k\in \mathbb{N}}$ such that $1<\gamma_k<+\infty$, $\gamma_k\to 1$ as $k\to \infty$, and for all $k\in \mathbb{N}$   
\[\lim_{N\to \infty}A_{\lfloor\gamma_k^{N}\rfloor}(\omega)=f(\omega)\]
for almost every $\omega\in\Omega$.
Then,
\[\lim_{N\to \infty} A_N(\omega)=f(\omega) \]
for almost every $\omega\in\Omega$.\label{lemma1}
\end{lemma}

The following notation will be used throughout the rest of this paper.
\begin{notation}
    Let $(a_n)_{n\in \N}$ and $(b_n)_{n\in \N}$ be positive real sequences. We write $a_n\ll b_n$ if and only if there exists $C>0$, independent of $n$, such that $a_n\leq Cb_n$ for all $n\in \N$. We also denote $a_n\asymp b_n$ if  $a_n\ll b_n$ and $b_n\ll a_n$.
\end{notation} 

The following is a version of the law of large numbers suitable for our context. The proof is inspired by the independent case established in \cite[Lemma A.6]{Frantzikinakis_Lesigne_Wierdl}. An analogous dynamical version of this result can be found in \cite[Theorem 2.3]{Viktoria_Xing2021}.

\begin{lemma}
Let $(\Omega,\mathcal{F},\mathbb{P})$ be a probability space and $(X_n)_{n\in \N}$ be a sequence of uniformly bounded and non-negative random variables defined on it. Suppose that there exist $a\in (0,1)$ and $c>0$ such that $\mathbb{E}_{\mathbb{P}}(X_n)=cn^{-a}$ and $\varepsilon>0$ such that
    \begin{equation*}
    \sum_{n=1}^N\sum_{m=n+1}^N\Cov(X_n,X_m)\ll N^{2-2a-\varepsilon}.
    \end{equation*}
    Then,
    \[\lim_{N\to \infty}\dfrac{1}{W_N}\sum_{n=1}^NX_n(\omega)=1\]
    for almost every $\omega\in \Omega$, where $W_N:=\sum_{n=1}^N \mathbb{E}_{\mathbb{P}}(X_n)$.\label{LLN}
\end{lemma}

\begin{proof}
Let $Y_n(\omega):=X_n(\omega)-\mathbb{E}_{\mathbb{P}}(X_n)$ and $\displaystyle A_N(\omega):=\dfrac{1}{W_N}\sum_{n=1}^N Y_n(\omega).$ It suffices to show that $A_N\to 0$ for almost every $\omega \in \Omega$. Note that
    \begin{equation}
    \mathbb{E}_{\mathbb{P}}(A_N^2)=\dfrac{1}{W_N^2}\sum_{1\leq n, m\leq N}\mathbb{E}_{\mathbb{P}}(Y_{n}Y_{m})=\dfrac{1}{W_N^2}\sum_{n=1}^{N}\mathbb{E}_{\mathbb{P}}(Y_n^2)+\dfrac{2}{W_N^2}\sum_{n=1}^{N}\sum_{m=n+1}^{N}\mathbb{E}_{\mathbb{P}}(Y_n Y_m).\label{eqn:LLN1}
\end{equation}
Since the $X_n$ are uniformly bounded, we have that $\mathbb{E}_{\mathbb{P}}(Y_n^2)=\mathbb{E}_{\mathbb{P}}(X_n^2)-\mathbb{E}_{\mathbb{P}}(X_n)^2\ll \mathbb{E}_{\mathbb{P}}(X_n).$ On the other hand,
$$\sum_{n=1}^N\sum_{m=n+1}^N\mathbb{E}_{\mathbb{P}}(Y_{n}Y_{m})=\sum_{n=1}^N\sum_{m=n+1}^N\Cov(X_{n},X_{m})\ll N^{2-2a-\varepsilon}.$$ 
Therefore,\begin{equation}
    \mathbb{E}_{\mathbb{P}}(A_N^2)\ll \dfrac{1}{W_N}+\dfrac{{N^{2-2a-\varepsilon}}}{W_N^2}.\label{eqn:LLN2}
\end{equation}
Since $\mathbb{E}_{\mathbb{P}}(X_n)=cn^{-a}$ and $\sum_{n=1}^N n^{-a}\asymp N^{1-a}$, we get $N^{1-a}\ll \sum_{n=1}^N\mathbb{E}_{\mathbb{P}}(X_n)=W_N$. We conclude from \eqref{eqn:LLN2} that
\begin{equation}
    \mathbb{E}_{\mathbb{P}}(A_N^2)\ll N^{-(1-a)}+N^{-\varepsilon}.\label{eqn:LLN3}
\end{equation}

For $k\in\N$, consider $\gamma_k:=1+1/k$. It follows from \eqref{eqn:LLN3} that for any $k\in \N$, 
\[\sum_{N\geq 1}\mathbb{E}_{\mathbb{P}}(A_{\lfloor\gamma_k^{N}\rfloor}^2)<+\infty.\]

Using Markov's inequality and the lemma of Borel-Cantelli, we conclude that for every $k\in \N$, 
\begin{equation}
    \lim_{N\to \infty}A_{\lfloor\gamma_k^{N}\rfloor}(\omega)=0\label{A_N}
\end{equation}
for almost every $\omega\in \Omega$. Notice that condition \eqref{lemma1_h} in \cref{lemma1} can be verified by noticing that $N^{-(1-a)} W_N $ has a positive limit as $N$ goes to infinity. \cref{lemma1} allows us to conclude by considering $(\gamma_k)_{k\in\N}$ and $f\equiv 0$.
\end{proof}

The following lemma, whose proof is elementary, is key to dealing with some obstacles that arise when trying to implement the method of Frantzikinakis, Lesigne, and Wierdl \cite{Frantzikinakis_Lesigne_Wierdl} in the non-independent case.

\begin{lemma} Let $(\Omega,\mathcal{F},\mathbb{P})$ be a probability space, and let $\mathcal{E}=(E_n)_{n\in \N}$ be a sequence of measurable sets in $\mathcal{F}$ of decreasing measure converging to zero and satisfying \eqref{P}. Let $X_n(\omega):=\mathbbm{1}_{E_n}(\omega)$, $\sigma_n:=\mathbb{E}_{\mathbb{P}}(X_n)$ and $Y_n(\omega):=X_n(\omega)-\sigma_n$. Then, for every $1\leq n_1\leq n_2\leq n_3\leq n_4$ such that $d:=n_2-n_1=n_4-n_3$, we can estimate the $4$-fold covariance as
\begin{align*}
    \mathbb{E}_{\mathbb{P}}\left( Y_{n_1} Y_{n_2} Y_{n_3} Y_{n_4} \right)&\ll  \sigma_{n_1}\rho(d)\Big(\sigma_{n_1}+\rho(n_3-n_2)\Big).
\end{align*}\label{key}
\end{lemma}

\begin{proof}
  Let $1\leq n_1\leq n_2\leq n_3\leq n_4$ be integers such that $d:=n_4-n_3=n_2-n_1$. By expanding $\mathbb{E}_{\mathbb{P}}(Y_{n_1}Y_{n_2}Y_{n_3}Y_{n_4})$  and grouping elements together, we can write
\begin{equation}
\begin{split}
\mathbb{E}_{\mathbb{P}}\left( Y_{n_1} Y_{n_2} Y_{n_3} Y_{n_4} \right) &=\Cov(X_{n_1},X_{n_2}X_{n_3}X_{n_4})-\sigma_{n_4}\Cov(X_{n_1},X_{n_2}X_{n_3})\\
&\quad-\sigma_{n_3}\Cov(X_{n_1},X_{n_2}X_{n_4})+\sigma_{n_3}\sigma_{n_4}\Cov(X_{n_1},X_{n_2})\\
&\quad-\sigma_{n_2}\Cov(X_{n_1},X_{n_3}X_{n_4})+\sigma_{n_2}\sigma_{n_4}\Cov(X_{n_1},X_{n_3})\\
&\quad+\sigma_{n_2}\sigma_{n_3}\Cov(X_{n_1},X_{n_4}).
\end{split}  \label{split1} 
\end{equation}
Since $\Cov(X_{n_1},X_{n_2}X_{n_3}X_{n_4})=\mathbb{P}(E_{n_1}\cap E_{n_2}\cap E_{n_3}\cap E_{n_4})-\mathbb{P}(E_{n_1})\mathbb{P}(E_{n_2}\cap E_{n_3}\cap E_{n_4})$ and the collection $\mathcal{E}$ satisfies \eqref{P},
\begin{align*}
    |\Cov(X_{n_1},X_{n_2}X_{n_3}X_{n_4})|&\leq \rho(n_2-n_1)\mathbb{P}(E_{n_2}\cap E_{n_3}\cap E_{n_4})\\
      &=\rho(n_2-n_1)\Big(\Cov(X_{n_2},X_{n_3}X_{n_4})+\sigma_{n_2}\mathbb{P}(E_{n_3}\cap E_{n_4})\Big) \\
      &\leq \rho(n_2-n_1)\Big(\rho(n_3-n_2)\mathbb{P}(E_{n_3}\cap E_{n_4})+\sigma_{n_2}\mathbb{P}(E_{n_3}\cap E_{n_4})\Big) \\
      &\leq \rho(n_2-n_1)\Big(\rho(n_3-n_2)+\sigma_{n_2}\Big)\Big(\rho(n_4-n_3)+\sigma_{n_3}\Big)\sigma_{n_4}\\
      &\leq \rho(d)\Big(\rho(n_3-n_2)+\sigma_{n_1}\Big)\Big(\rho(d)+\sigma_{n_1}\Big)\sigma_{n_1},
\end{align*}
where in the last step we used that $(\sigma_n)_{n\in \N}$ is decreasing and $n_2-n_1=n_4-n_3=d$. Using that $\rho(n)$ and $\sigma_n$ are decreasing to 0, we conclude 
\[ |\Cov(X_{n_1},X_{n_2}X_{n_3}X_{n_4})|\ll \rho(d)\rho(n_3-n_2)\sigma_{n_1}+\rho(d)\sigma_{n_1}^2.\]Similarly, we obtain
\begin{equation}
\begin{split}
    &|\sigma_{n_4}\Cov(X_{n_1},X_{n_2}X_{n_3})|\leq \rho(d)\Big(\rho(n_3-n_2)+\sigma_{n_1}\Big)\sigma_{n_1}^2,\\
    &|\sigma_{n_3}\Cov(X_{n_1},X_{n_2}X_{n_4})|\leq \rho(d)\Big(\rho(n_4-n_2)+\sigma_{n_1}\Big)\sigma_{n_1}^2, \text{ and }\\
    &|\sigma_{n_2}\Cov(X_{n_1},X_{n_3}X_{n_4})|\leq \rho(n_3-n_1)\Big(\rho(d)+\sigma_{n_1}\Big)\sigma_{n_1}^2.\label{split3} 
\end{split}
\end{equation}
As $\rho$ is decreasing, each of the terms in \eqref{split3} can be bounded above by $$\rho(d)\rho(n_3-n_2)\sigma_{n_1}^2+\rho(d)\sigma_{n_1}^3.$$
Also, note that
\begin{equation}
\begin{split}
    &|\sigma_{n_3}\sigma_{n_4}\Cov(X_{n_1},X_{n_2})|\leq \rho(d)\sigma_{n_1}^3,\\
    &|\sigma_{n_2}\sigma_{n_4}\Cov(X_{n_1},X_{n_3})|\leq \rho(n_3-n_1)\sigma_{n_1}^3,\text{ and }\\
    &|\sigma_{n_2}\sigma_{n_3}\Cov(X_{n_1},X_{n_4})|\leq \rho(n_4-n_1)\sigma_{n_1}^3. \label{split4} 
\end{split}
\end{equation}
Hence, each term in \eqref{split4} can be bounded above by $\rho(d)\sigma_{n_1}^3.$
Replacing all the estimations in \eqref{split1}, it follows that
\begin{align*}
    \mathbb{E}_{\mathbb{P}}\left( Y_{n_1} Y_{n_2} Y_{n_3} Y_{n_4} \right)&\ll \rho(d)\rho(n_3-n_2)\sigma_{n_1}+\rho(d)\sigma_{n_1}^2\\
    &+\rho(d)\rho(n_3-n_2)\sigma_{n_1}^2+\rho(d)\sigma_{n_1}^3+\rho(d)\sigma_{n_1}^3\\
    &\ll \rho(d)\rho(n_3-n_2)\sigma_{n_1}+\rho(d)\sigma_{n_1}^2.
\end{align*}\end{proof}

\section{Proof of the main result and consequences}
The following is our main result. In this section, we will prove it, and as a consequence, we will derive \cref{thmA}. 

\begin{theorem}
    Let $(\Omega,\mathcal{F},\mathbb{P})$ be a probability space and $(E_n)_{n\in \N}\subseteq \mathcal{F}$ be a sequence satisfying \eqref{P} with rate $\rho$ such that $\sum_{n\in\N}\rho(n)<+\infty$, and suppose there exist $a\in (0,1/2)$ and $c>0$ such that $\mathbb{P}(E_n)=cn^{-a}$ for all $n\in\N$. For $\omega\in \Omega$, we define  $\kappa_1(\omega):=\inf\{k\in \N:\ \omega\in E_k\}$ and for all $n>1$,
    \[\kappa_n(\omega):=\inf\left\{k>\kappa_{n-1}(\omega): \omega\in E_{k} \right\}.\]
    Then, for almost every $\omega\in \Omega$, the following holds: For every measure-preserving system $(X,\mathcal{X},\mu,T)$ and $f\in L^{2}(\mu)$,
    \begin{equation}
        \lim_{N\to\infty}\dfrac{1}{N}\sum_{n=1}^Nf(T^{\kappa_n(\omega)}x)=\mathbb{E}_{\mu}(f|\mathcal{I}(T))(x)\label{teo_eq}
    \end{equation}
   for almost every $x\in X$. That is, for almost every $\omega\in \Omega$, the sequence $(\kappa_n(\omega))_{n\in \N}$ is pointwise universally $L^2$-good and ergodic.\label{main} 
\end{theorem}
It is important to note that the definition of the sequence $\kappa_n(\omega)$ is purely probabilistic and does not involve any dynamics.  For clarity, in the statement of the theorem and its proof, we use a subscript on the symbol $\E$ to indicate the probability space for which we are computing the expectation or conditional expectation.

To prove \cref{main}, we need the following version of the classical Van der Corput lemma, as stated in \cite[Lemma A.4]{Frantzikinakis_Lesigne_Wierdl}.

\begin{lemma}[Van der Corput] Let $V$ be a space with inner product $\langle\cdot, \cdot \rangle$, $N\in \mathbb{N}$ and $v_1,...,v_N\in V$. Then, for all integers $M$ between 1 and N we have
\[\left\| \sum_{n=1}^{N} v_n\right\|^2\leq 2M^{-1}N \cdot \sum_{n=1}^{N}\|v_n\|^2+4M^{-1}N\sum_{m=1}^{M}\left |\sum_{n=1}^{N-m} \langle v_{n+m},v_n \rangle\right |,\]
where the norm is the one induced by the inner product.
\label{vdc}
\end{lemma}

\begin{proof}[Proof of \cref{main}] 
Let $X_n:=\mathbbm{1}_{E_n}$, $\sigma_n=\mathbb{E}_{\P}(X_n)(=\mathbb{P}(E_n))$, $\mathbf{W_N}(\omega):=\sum_{n=1}^N X_n(\omega)$ and $W_N:=\sum_{n=1}^N \sigma_n$. The following simple estimates will be used repeatedly:
 $$\sum_{n=1}^N\sigma_n\ll N^{1-a} \text{ and }  \sum_{n=1}^N\sigma_n^2\ll N^{1-2a}.$$

Recalling property \eqref{P} and using that $\rho(n)$ is summable, we get that
    \[\sum_{n=1}^N\sum_{m=n+1}^N\Cov(X_n,X_m)\ll \sum_{n=1}^N\sum_{m=n+1}^N \rho(m-n)\sigma_m\ll N^{1-a}= N^{2-2a-\delta},\]
    for $\delta:=1-a>0$. Thus, from \cref{LLN},

\begin{equation}
    \displaystyle\lim_{N\to \infty} \mathbf{W_N
}(\omega)/W_N=1\text{ for almost every $\omega\in\Omega$}.\label{AS_1}
\end{equation} This shows that the sequence $(\kappa_n(\omega))_{n\in\N}$ is well-defined for a.e. $\omega\in \Omega$. From now, we restrict ourselves to $\omega\in \Omega$ where \eqref{AS_1} holds. Let $(X,\mathcal{X},\mu,T)$ be a m.p.s. and $f\in L^2(\mu)$. By definition of $\kappa_n(\omega)$, it is clear that the sequences
\begin{equation} 
      \dfrac{1}{N}\sum_{n=1}^Nf(T^{\kappa_n(\omega)}x)\label{reduction_11}
  \end{equation}
  and    
    \begin{equation} 
      \dfrac{1}{\mathbf{W_N(\omega)}}\sum_{n=1}^NX_n(\omega)f(T^nx)\label{reduction_1}
  \end{equation} 
have exactly the same asymptotic behavior (meaning that the limit of  \eqref{reduction_11} exists if and only if the limit of \eqref{reduction_1} exists, and in such a case, the limits coincide).
Therefore, using \eqref{AS_1}, we can reduce the problem to show that \begin{equation}
    \lim_{N\to\infty}\dfrac{1}{W_N}\sum_{n=1}^NX_n(\omega)f(T^nx)=\mathbb{E}_{\mu}(f|\mathcal{I}(T))(x)\label{eqn:main1}
  \end{equation}
  for almost every $x\in X$. On the other hand, a direct application of \cite[Lemma A.3]{Frantzikinakis_Lesigne_Wierdl} allows us to deduce that the asymptotic behavior of
  
  \[\dfrac{1}{W_N}\sum_{n=1}^N\sigma_nf(T^nx)\quad \text{and }\quad \dfrac{1}{N}\sum_{n=1}^Nf(T^nx)\]
  is the same. Thus, as a consequence of the Birkhoff ergodic theorem, it only remains to prove that
  \begin{equation}
      \lim_{N\to\infty}\dfrac{1}{W_N}\sum_{n=1}^N(X_n(\omega)-\sigma_n)f(T^nx)=0 \quad \text{ for almost every $x\in X$.}\label{eqn:main2}
  \end{equation}
  
  Let $Y_n(\omega):=X_n(\omega)-\sigma_n
   $ and $b:=2a+\varepsilon$, where $\varepsilon\in (0,1-2a)$. Applying \cref{vdc} with $M=\left \lfloor N^b \right\rfloor $ and $v_n=Y_n(\omega)T^nf$, we get 
    \begin{equation*}
          \begin{split}
          V_N:=\left\|N^{-1+a}\sum_{n=1}^N Y_n(\omega)\cdot T^{n}f \right\|^2_{L^2(\mu)}&\ll N^{-1+2a} M^{-1}\sum_{n=1}^{N}\| Y_n(\omega)\cdot T^{n}f\|_{L^2(\mu)}^2+V_{1,N},
    \end{split}
    \end{equation*}
where \[V_{1,N}:=N^{-1+2a}M^{-1}\sum_{m=1}^{M}\left|\sum_{n=1}^{N-m}\int_{X} Y_{n+m}(\omega)\cdot Y_n(\omega)\cdot T^{n+m}f\cdot T^n f\text{d}\mu\right|.\]

Notice that $$\sum_{n=1}^{N}\| Y_n(\omega)\cdot T^{n}f\|_{L^2(\mu)}^2\leq \|f\|^2_{L^2(\mu)}\sum_{n=1}^NY^2_n(\omega).$$ Since $Y_n^2\leq X_n+\sigma_n^2$, we have  that $\sum_{n=1}^N \mathbb{E}_{\mathbb{P}}(Y^2_n)\ll N^{1-a}$. Thus, \begin{equation}
    \mathbb{E}_{\mathbb{P}}(V_N)\ll N^{a-b}+\mathbb{E}_{\mathbb{P}}(V_{1,N}).\label{eqn:main3}
\end{equation}

Composing with $T^{-n}$ and using Cauchy-Schwarz, we obtain that
\begin{equation}
    V_{1,N}\ll N^{-1+2a}M^{-1}\sum_{m=1}^{M}\left|\sum_{n=1}^{N-m}Y_{n+m}Y_n\right|.\label{eqn:main4}
\end{equation}

Note that by convexity (of $x\mapsto x^2$) and Jensen's inequality, \begin{align*}
\E_{\P}(V_{1,N})^2&\ll N^{4a-2}\Big (\E_{\P} \Big(\frac{1}{M}\sum_{m=1}^M \Bigl|\sum_{n=1}^{N-m} Y_{n+m}Y_{n} \Bigr |\Big)\Big)^2\\ 
&\ll N^{4a-b-2} \E_{\P} \Big(\sum_{m=1}^M \Bigl|\sum_{n=1}^{N-m} Y_{n+m}Y_{n} \Bigr |^2 \Big).
\end{align*}

Expanding the square in the last expression, we get that $\E_{\P}(V_{1,N})^2$ can be bounded by the sum of the terms 

\begin{equation}
    N^{4a-b-2}\sum_{m=1}^M \sum_{n=1}^{N-m}  \E_{\P} (Y_n^2Y_{n+m}^2)
\label{eq:term1}\end{equation}
and 
\begin{equation}
N^{4a-b-2}\sum_{m=1}^M \sum_{n_1\neq n_2}^{N-m} \mathbb{E}_{\P}(Y_{n_1}Y_{n_1+m}Y_{n_2}Y_{n_2+m}). \label{eq:term2}
\end{equation}

Firstly, recalling that the variables $X_n$ take values in $\{0,1\}$, note that 
 \begin{align*}
     Y_n^2Y_{n+m}^2&=X_{n}X_{n+m}-2X_nX_{n+m}\sigma_{n+m}+X_n\sigma_{n+m}^2\\
     &-2\sigma_nX_nX_{n+m}+4X_nX_{n+m}\sigma_n\sigma_{n+m}-2X_n\sigma_n\sigma_{n+m}^2\\
     &+X_{n+m}\sigma_n^2-2\sigma_{n}^2\sigma_{n+m}X_{n+m}+\sigma_{n}^2\sigma_{n+m}^2\\
     &\ll X_nX_{n+m}(1-2\sigma_{n+m}-2\sigma_n+4\sigma_{n}\sigma_{n+m})+\sigma_n^2.
 \end{align*}
It follows from the previous inequality and  property \eqref{P} that \eqref{eq:term1} is bounded by a constant times 
\begin{align*}
    &N^{4a-b-2}\sum_{m=1}^M \sum_{n=1}^{N-m} \Big(\mathbb{E}_{\mathbb{P}}(X_nX_{n+m})+\sigma_n^2\Big)\\
    \ll & N^{4a-b-2}\sum_{m=1}^M \sum_{n=1}^{N-m}\Big((\rho(m)+\sigma_n)\sigma_{n+m}+\sigma_n^2\Big)\\
    \ll & N^{4a-b-2}\Big(\sum_{m=1}^M \rho(m)\sum_{n=1}^{N-m}\sigma_{n+m}+\sum_{m=1}^M\sum_{n=1}^{N-m}2\sigma_n^2 \Big)\\ 
    \ll&   N^{4a-b-2} N^{1-2a+b}= N^{2a-1}. 
\end{align*}
Using \cref{key} and distinguishing between the cases $n_1\leq n_1+m\leq n_2\leq n_2+m$ and $n_1\leq n_2\leq n_1+m\leq n_2+m$, we can bound
\[N^{4a-b-2}\sum_{m=1}^M \sum_{n_1\neq n_2}^{N-m} \mathbb{E}_{\P}(Y_{n_1}Y_{n_1+m}Y_{n_2}Y_{n_2+m})\ll \alpha_{1,N}+\alpha_{2,N},\] where
\begin{align*}
\alpha_{1,N}&:=N^{4a-b-2} \sum_{m=1}^M \sum_{n_1=1}^{N-m}\sum_{n_2=n_1+m+1}^{N-m}\sigma_{n_1}\rho(m)\Big(\sigma_{n_1}+\rho(n_2-n_1-m)\Big),\\
\alpha_{2,N}&:=N^{4a-b-2} \sum_{m=1}^M\sum_{n_1=1}^{N-m}\sum_{n_2=n_1+1}^{n_1+m}\sigma_{n_1}\rho(n_2-n_1)\Big(\sigma_{n_1}+\rho(n_1+m-n_2)\Big).
\end{align*}

Since $\rho$ is summable, we can bound
\[ \alpha_{1,N} \ll N^{4a-2-b}\Big(N^{2-2a}+N^{1-a}\Big)\ll N^{2a-b}.\]
On the other hand, notice that 
 \[N^{4a-b-2} \sum_{m=1}^M\sum_{n_1=1}^{N} \sigma_{n_1} \sum_{t=1}^{m} \rho(t)\rho(m-t)  \ll N^{4a-b-2} N^{1-a} = N^{3a-b-1},\]
 where we have used that $\sum_{m=1}^M \sum_{t=1}^{m} \rho(t)\rho(m-t)=\sum_{t=1}^M\rho(t)\sum_{m=t}^M \rho(m-t)\ll 1$.
 Finally, we have
\[ N^{4a-b-2} \sum_{m=1}^M\sum_{n_1=1}^{N-m}\sigma_{n_1}^2\sum_{n_2=n_1+1}^{n_1+m} \rho(n_2-n_1)\ll N^{4a-b-2}N^{1-2a+b}=N^{2a-1}.\]
Summing up, 

\begin{equation}
\begin{split}
     \E_{\P}(V_N)^2&\ll N^{2(a-b)}+N^{2a-1}+N^{2a-b}+N^{3a-b-1}+N^{2a-1}\\
    &= N^{-2(a+\varepsilon)}+2N^{2a-1}+N^{-\varepsilon}+N^{a-\varepsilon-1}\\
    &\ll N^{-\varepsilon}.\label{final_estimation}
\end{split}   
\end{equation}

It follows that for $\gamma_k:=1+1/k$, $$\sum_{N=1}^\infty \mathbb{E}_{\mathbb{P}}(V_{\lfloor\gamma_k^N\rfloor})<+\infty.$$ 
Therefore, for every $k\in \N$, there is a set of full measure $\Omega_k\subseteq \Omega$ (where \eqref{AS_1} also holds) such that
$$\sum_{N=1}^\infty V_{\lfloor\gamma_k^N\rfloor}=\sum_{N=1}^\infty\left\|\lfloor\gamma_k^N\rfloor^{-1+a}\sum_{n=1}^{\lfloor\gamma_k^N\rfloor} Y_n(\omega)\cdot T^{n}f \right\|^2_{L^2(\mu)}<+\infty \text{ for almost every $\omega\in \Omega_k$}.$$

It is important to note that the notion of almost surely in the above estimates does not depend on the dynamical system $(X,\mathcal{X},\mu,T)$. From a standard application of the Markov inequality and Borel-Cantelli we obtain that for every $k\in \N$ and $\omega\in \Omega':=\bigcap_{k\in\N} \Omega_k$,
\[\lim_{N\to \infty}\dfrac{1}{W_{\lfloor\gamma_k^N\rfloor}}\sum_{n=1}^{\lfloor\gamma_k^N\rfloor} Y_n(\omega)\cdot f(T^{n}x)=0 \]
for almost every $x\in X$. An application of \cref{lemma1} completes the proof.
\end{proof}

\begin{remark}
The summable condition on $\rho$  can be replaced by $\rho$ such that for all $\varepsilon>0$, $\sum_{n=1}^N\rho(n)/N^\varepsilon\to 0$ as $N\to \infty$. This includes, for example, $\rho(n)=1/n$.
\end{remark}
Notice that the sequence $(\kappa_{n}(\omega))_{n\in\N}$ in \cref{main} is equivalent to 
the random sequence of integers $(a_n(\omega))_{n\in\N}$ generated by the random variables $X_n(\omega):=\mathbbm{1}_{E_n}(\omega)$ when the sets $(E_n)_{n\in \N}$ are independent. Hence, the previous result can be understood as an extension of \cref{thm:random_ergodic_averages} for non-independent sequences, but restricted to $a\in (0,1/2)$ and $p=2$. In this direction, it is natural to ask the following:

\begin{question} Does \cref{main} hold for $a\in (0,1)$ and $p\in (1,2]$?    
\end{question}

Similarly, we could ask if the sequence $(\kappa_n(\omega))_{n\in\N}$ of \cref{thm:random_ergodic_averages} is pointwise universally $L^1$-good. Notice that, for $a=1/2$, the random sequence can be understood as a random version of the sequence $n^2$. It is known that the sequence $n^2$ is not $L^1$-good for pointwise convergence (see \cite{Buczolich_Mauldin, Buczolich_Mauldin_2}). On the other hand, in the independent framework, the random sequence is universally $L^1$-good for $a\in (0,1/2)$ \cite{LaVictoire}. This provides positive evidence for the next question. 

\begin{question}
    Let $a\in (0,1/2)$. Under the conditions of \cref{main}, is the sequence $(\kappa_n(\omega))_{n\in\N}$ pointwise universally $L^1$-good for almost every $\omega$?
\end{question}

Now we focus on retrieving \cref{thmA} as a consequence of \cref{main}.

\begin{proof}[Proof of \cref{thmA}]
    For every $n\in \N$, we define $E'_n:=S^{-n}E_n$. Note that $\nu(E'_n)=\nu(E_n)=cn^{-a}$. Now, we show that $(E'_n)_{n\in\N}$ satisfies condition \eqref{P} with rate $\rho$. 
    
    Let $1\leq n_1<\cdots<n_k$. Note that
    \begin{equation*}
    \begin{split}
             \nu(E'_{n_1}\cap \cdots\cap E'_{n_k})=\nu\Big(E_{n_1}\cap S^{-(n_2-n_1)}(E_{n_2}\cap S^{-(n_3-n_2)}E_{n_3}\cap\cdots\cap S^{-(n_k-n_2)}E_{n_k})\Big).
    \end{split}
    \end{equation*}
    By denoting $A:=E_{n_2}\cap S^{-(n_3-n_2)}E_{n_3}\cap\cdots\cap S^{-(n_k-n_2)}E_{n_k}$ and $g:=\mathbbm{1}_{A
    }$, we can write
    \begin{equation}
        \nu(E'_{n_1}\cap \cdots\cap E'_{n_k})=\int_Y \mathbbm{1}_{E_{n_1}}\cdot g\circ S^{n_2-n_1}d\nu.\label{eqn:cor_1}
    \end{equation}
    Also, notice that \[\int_Y gd\nu=\int_Y g\circ S^{n_2}d\nu=\nu(S^{-n_2}E_{n_2}\cap S^{-n_3}E_{n_3}\cap\cdots\cap S^{-n_k}E_{n_k}).\]
    Since $g\in L^1(\nu)$ is supported on $E_{n_2}$ and $(E_n)_{n\in\N}$ holds a decay of correlations against $L^1(\nu)$ (\cref{def:correlations}), there exists $C>0$ such that
    \begin{equation}
    \begin{split}
    \left| \nu(E'_{n_1}\cap \cdots\cap E'_{n_k})- \nu(E'_{n_1})\nu(E'_{n_2}\cap \cdots\cap E'_{n_k})\right|&=
           \left|\int_{E_{n_1}} g\circ S^{n_2-n_1}d\nu-\nu(E_{n_1})\int_{Y}gd\nu\right|\\
           &\leq C\cdot \rho(n_2-n_1)\|g\|_{L^1}\\
        &=C\rho(n_2-n_1)\nu(E'_{n_2}\cap\cdots\cap E'_{n_k}).
    \end{split}
     \label{eqn:cor_2}
    \end{equation}
    We conclude that the sequence $(E'_n)_{n\in\N}$ satisfies \eqref{P} with summable rate $\rho$. 
    Since $(r_{n}(\omega,\mathcal{E}))_{n\in\N}$ is exactly the sequence $(\kappa_n(\omega))_{n\in\N}$ associated with $(E'_n)_{n\in\N}$, the result follows from  \cref{main}. 
\end{proof}

   It is worth mentioning that \cref{thmA} (similarly for \cref{main}) does not hold for sequences of sets $(E_n)_{n\in\N}$ with $\nu(E_n)=n^{-a}$ in an arbitrary system $(Y,\mathcal{Y},\nu,S)$, even if it is ergodic\footnote{In fact, without mixing assumptions, it is not possible to ensure that the sequence of return times of $(r_n(y,\mathcal{E}))_{n\in\N}$ is well-defined almost surely.}. For instance, let $$(Y,\mathcal{Y},\nu,S)=(X,\mathcal{X},\mu,T)=([0,1),\mathcal{B}([0,1)),\lambda,T_\alpha),$$ where $\lambda$ is the Lebesgue measure and $T_\alpha$ is the irrational rotation by some badly approximable\footnote{$\alpha\in (0,1)$ is badly approximable if there exists $c>0$ such that $|\alpha-p/q|>c/q^2$ for all $p,q\in \N$.} $\alpha\notin \mathbb{Q}$. If we take $\mathcal{E}=(E_n)_{n\in\N}$ where $E_n:=B(0,\frac{n^{-a}}{2})\subseteq [0,1)$, it follows from \cite[Theorem 1.2]{Chaika_Constantine} that for almost every $y\in [0,1)$, $T_\alpha^ny\in E_n$ for infinitely many $n\in \N$ (see also \cite{Fayad,Kim_Rotations,Kurzweil} for related results). This shows that the sequence of return times $(r_n(y,\mathcal{E}))_{n\in\N}$ is well-defined almost surely. Note that for all $x\in [0,1)$, $T_\alpha^{r_n(y,\mathcal{E})}x \in E_n + x-y$. So, for $f=\mathbbm{1}_{[0,1/4]}$ and any $x\in [y+1/2,y+3/4]$, $f(T_\alpha^{r_n(y,\mathcal{E})}x)=0$ for all $n$ large enough. Hence, \cref{thmA} cannot hold for this type of system. Therefore, the sets $(E_n)_{n\in \N}$ must satisfy some type of mixing assumption; as we have shown, the decay of correlations is sufficient. One could, of course, ask whether the result can be obtained under weaker mixing conditions.  Also note that the system $(Y,\mathcal{Y},S,\nu)$ does not need to be mixing itself, as the subsets $E_n$ may be measurable with respect to a proper factor of it.

\begin{question}
    Can we construct interesting universally good sequences of return times $(r_n)_{n\in\N}$ in some more general classes of systems $(Y,\mathcal{Y},\nu,S)$? 
\end{question}

\subsection{Applications to transformations on the interval}
Regarding applications of \cref{thmA} to particular measure-preserving systems, many can be derived from systems with rapidly mixing properties. For a function $f\colon [0,1]\subseteq \R\to \R$, recall that 
 $$\|f\|_{BV}:=\displaystyle \|f\|_{\infty}+\sup_{0=x_0<x_1<\cdots<x_m=1}\sum_{i=0}^{m-1}|f(x_{i+1})-f(x_i)|,$$ where the supremum is taken over all the finite sequences $0=x_0<x_1<\cdots<x_m=1$ of $[0,1]$. We say that $f$ has bounded variation if $\|f\|_{BV}<\infty$.  Note that multidimensional versions of this notion can also be defined (for instance, see \cite{Liverani_Multimensional}).

\begin{definition}
    Let $([0,1],\mathcal{B}([0,1]),\nu,S)$ be a m.p.s. We say that the system has a decay of correlations for functions of bounded variation against $L^1(\nu)$ with rate $\rho\colon\N\to \mathbb{R}_{+}$ if there exists $C>0$ such that
    \[\left|\int_{[0,1]} f\cdot g\circ S^nd\nu-\int_{[0,1]} fd\nu \cdot \int_{[0,1]} gd\nu \right|\leq C\rho(n)\|f\|_{BV}||g||_{L^{1}}\]
    for all $n\in\N$, $f$ of bounded variation and $g\in L^1(\nu)$. \label{def:correlationsBV}
\end{definition} 

\begin{corollary}
  Let $([0,1],\mathcal{B}([0,1]),\nu,S)$ be a system that holds a decay of correlations for functions of bounded variation against $L^1(\nu)$ with a decreasing and summable rate $\rho$, and let $(I_n)_{n\in \N}\subset \mathcal{B}([0,1])$ such that:
  \begin{enumerate}
      \item There exists $\ell \in \N$ such that each $I_n$ is union of at most $\ell$ intervals, and
      \item there exist $a\in (0,1/2)$ and $c>0$ such that $\nu(I_n)=cn^{-a}$.
  \end{enumerate}
  Then, the sequence \[\{n\in \N:\ S^ny\in I_n\}\]
  is pointwise universally $L^2$-good and ergodic for almost every $y\in [0,1]$.\label{cor}
\end{corollary}
\begin{proof}
     Notice that $\|\mathbbm{1}_{I_n}\|_{BV}$ is uniformly bounded as each $I_n$ is a union of at most $\ell$ intervals. Therefore, the sequence $(I_n)_{n\in\N}$ has a decay of correlation against $L^1(\nu)$ with rate $\rho$. The conclusion follows from  \cref{thmA}.
\end{proof}

As a consequence, \cref{corA} and \cref{corB} are obtained.
\medskip

Finally, in this work, we focus on single ergodic averages, but we believe that the independence condition can also be partially removed when studying versions of random ergodic averages involving more transformations. Specifically, we will refer to semi-random averages: Let $a\in (0,1/14)$ and $(X_n)_{n\in\N}$ be a sequence of independent random variables, taking values in $\{0,1\}$, such that $\mathbb{P}(X_n=1)=n^{-a}$. In \cite{Frantzikinakis_Lesigne_Wierdl}, it was shown that for almost every $\omega\in \Omega$: For any probability space $(X,\mathcal{X},\mu)$ and commuting measure-preserving transformations $T_1,T_2\colon X\to X$, it holds that for any $f_1,f_2\in L^{\infty}(\mu)$, for almost every $x\in X$,
\begin{equation}\lim_{N\to\infty}\dfrac{1}{N}\sum_{n=1}^Nf_1(T_1^nx)\cdot f_2(T_2^{a_n(\omega)}x)=\mathbb{E}_\mu(f_1|\ \mathcal{I}(T_1))(x)\cdot\mathbb{E}_\mu(f_2|\ \mathcal{I}(T_2))(x),
    \label{semi_random}
\end{equation}
where $(a_n(\omega))_{n\in\N}$ denotes the random sequence defined in \cref{def:random_sequence}.
Those averages are called \emph{semi-random ergodic averages} (see also \cite{Frantzikinakis_Lesigne_Wierdl_Szemeredi} for the case $a\in (0,1/2)$ when $T_1=T_2$). Thus, we naturally ask for an analogous result to \cref{cor} in the context of semi-random averages. As a first approach, we pose the following question:
\begin{question}
    Let $a\in (0,1/2)$ and an integer $p\geq 2$. Is the sequence \begin{equation}
        \{n\in \N:\ p^ny\mod{1}\in (0,n^{-a})\}\label{px}
    \end{equation} good for the pointwise convergence of averages \eqref{semi_random} when considering functions in $L^\infty(\mu)$ for almost every $y\in [0,1]$?\label{Q:semi_random}
\end{question}
 It is worth mentioning that semi-random averages in the non-independent case have been partially addressed by the third author in \cite{Tesis}, but with a probabilistic approach that does not seem to contain the dynamical examples of interest in this paper. We are currently exploring an extension of \cref{thmA,main} to the setting of semi-random averages.

We can also ask about the sequence of 2-random ergodic averages such as
 \begin{equation}
     \dfrac{1}{N}\sum_{n=1}^Nf_1(T_1^{a_n(\omega)}x)\cdot f_2(T_2^{a_n(\omega)}x).\label{2random}
 \end{equation}
The convergence of those averages is known when $a\in (0,1/2)$  \cite{Frantzikinakis_Lesigne_Wierdl} (see also \cite{Frantzikinakis_Lesigne_Wierdl_Szemeredi}). Finally, we ask about the analogous version of  \cref{Q:semi_random} for these averages.
\begin{question}
    Is the sequence defined in \eqref{px} good for the pointwise convergence of the averages \eqref{2random} when considering functions in $L^\infty(\mu)$ for almost every $y\in [0,1]$?
\end{question}

\bibliographystyle{abbrv}
\bibliography{bibliography,refs}

\end{document}